\documentclass[12pt]{amsart}
\usepackage{amsfonts,amsmath,amssymb,mathrsfs}
\usepackage{amsthm}
\usepackage{graphicx}
\usepackage{url}
\usepackage[colorlinks=true, allcolors=blue]{hyperref}
\usepackage{diagbox}

\newtheorem{theorem}{Theorem}

\newtheorem{corollary}[theorem]{Corollary}

\newtheorem{example}[theorem]{Example}
\newtheorem{conjecture}[theorem]{Conjecture}

\numberwithin{theorem}{section}
\numberwithin{equation}{section}
\numberwithin{table}{section}

\newcommand{\Z}{\mathbb{Z}}
\newcommand{\Q}{\mathbb{Q}}

\newcommand{\Res}{\mathrm{Res}}
\newcommand{\cont}{\mathrm{cont}}

\newcommand{\wH}{\mathrm{H}}

\begin{document}

	\title[B\'{e}zout's identity and resultants]{Three integers arising from B\'{e}zout's identity and resultants of integer polynomials}
	
	\author{Zhiqian Liu}
	\address{School of Mathematical Sciences, South China Normal University, Guangzhou, 510631, China}
	\email{20222231028@m.scnu.edu.cn}
	
	\author{Xiaoting Li}
	\address{School of Mathematical Sciences, South China Normal University, Guangzhou, 510631, China}
	\email{20222221038@m.scnu.edu.cn}
	
	\author{Wenheng Liu}
	\address{School of Mathematics and Statistics, Central South University, Changsha, 410083, China}
	\email{whliu1023@csu.edu.cn}

	\author{Min Sha}
	\address{School of Mathematical Sciences, South China Normal University, Guangzhou, 510631, China}
	\email{min.sha@m.scnu.edu.cn}

	

\begin{abstract}
		In this paper, we study three integers arising naturally from B\'{e}zout's identity, the resultant and the reduced resultant of two coprime integer polynomials. We establish several new divisibility relations among them. We also pose two conjectures by making computations. 
	\end{abstract}
	
	\maketitle

	\section{Introduction}
	In this paper, integer polynomials are meant to be polynomials in $\Z[x]$, that is, polynomials in the variable $x$ and with integer coefficients.
	They are definitely among the fundamental objects in algebra and number theory. 
	The research around their arithmetic properties has a very long history and is still very active. 
	The arithmetic properties include reducibility, degeneracy, decomposability, moduli of roots, signature of roots,  multiplicative dependence of roots, etc. See \cite{Aki, BHP, BEGRS, BBBCM, BDFPS, BDPS, Du2014, Du2016, Du2018,DS1, DS2, DS3, DS4, DS5, DS6, HTV, Kuba, ZYZ} for some recent developments. 
	
	In this paper, we study the relations among three integers arising naturally from B\'{e}zout's identity, the resultant and the reduced resultant of two coprime integer polynomials with positive degree. 
	Here, we say two integer polynomials are coprime if they don't have a common root. 
	
	From now on, $f(x)$ and $g(x)$ are two coprime integer polynomials of positive degree. 
	Below we define the three integers studied in this paper with respect to $f(x)$ and $g(x)$.
	
	By B\'ezout's identity, we know that there exist uniquely two polynomials $p(x)$ and $q(x)$ in $\Q[x]$ such that:
	\begin{equation} \label{eq:Bezout}  
		p(x)f(x)+q(x)g(x) = 1, \quad \deg (p) < \deg (g), \quad \deg (q) < \deg (f).
	\end{equation}
	Then, we define the integer $B(f, g)$ to be the least common multiple of the denominators of all the coefficients  of $p(x)$ and $q(x)$ in \eqref{eq:Bezout} (the coefficients are reduced to their lowest terms). Letting $B=B(f,g)$, we have 
	\begin{equation}\label{eq:uvB}
		Bp(x)f(x) + Bq(x)g(x) = B, \quad Bp(x), Bq(x) \in \Z[x].
	\end{equation}
	
	Since $f(x)$ and $g(x)$ are non-constant and coprime, their resultant $\Res(f,g)$ is a non-zero integer. 
	Then, we define the integer $R(f,g) = |\Res(f,g)|$. It is well known that the resultant $\Res(f,g)$ 
	can be expressed as a linear combination of $f(x)$ and $g(x)$ over the polynomial ring $\Z[x]$; 
	see \cite[page 157, Proposition 9]{Cox}. 
	
	In addition, we define the integer $r(f, g)$ to be  the smallest positive integer 
	in the following set 
	$$
	f(x)\Z[x] + g(x)\Z[x] = \left\{p(x)f(x)+q(x)g(x):\,p(x),q(x)\in\mathbb{Z}[x] \right\}.
	$$
	This integer is called \textit{reduced resultant} in \cite{Pohst}. 
	In addition, it is also called \textit{congruence number} in \cite[Definition 2.6]{TW}, 
	and moreover, the authors in \cite{TW} also described its applications.
	
	By definition, $r(f,g)$ is the smallest positive integer  in the intersection 
	$(f(x)\Z[x] + g(x)\Z[x]) \cap \Z$, which in fact is an ideal of the integer ring $\Z$ and so is equal to $r(f,g)\Z$. 
	Thus, by definition we directly have 
	\begin{equation} \label{eq:rBR}
		r(f,g) \mid B(f,g) \quad \textrm{and} \quad r(f,g) \mid R(f,g).
	\end{equation}
This was known previously. 
	
	Moreover, for the ideal generated by $f(x)$ and $g(x)$ in $\Z[x]$, if we compute its Gr\"{o}bner basis (see \cite{Cox}), then the constant polynomial in this basis is exactly $r(f,g)$. 
	In this paper, we compute $r(f,g)$ based on this fact and using Mathematica. 
	
	From the discussions in MathOverflow (see, for instance, \cite{Kana, Rug, Voloch}), 
	one can see that many people are interested with the relations among 
	the three integers $B(f, g)$, $r(f, g)$ and $R(f, g)$. 
	However, so far we haven't found a paper studying their relations systematically. 
	
	In this paper, we want to investigate systematically the divisibility relations among the three integers $B(f, g)$, $r(f, g)$ and $R(f, g)$. We believe that some of them are new; such as Theorems~\ref{thm:Bdr}, \ref{thm:rR}, \ref{thm:BR1} and \ref{thm:BR2}. We also pose two conjectures at the end. 

To avoid systematic repetitions, throughout the paper we fix the following assumptions and notations: 
\begin{itemize}
\item $f(x)$ and $g(x)$ are two coprime integer polynomials of positive degree; 

\item for every $h \in \Z[x]$, $L(h)$ is the leading coefficient of $h$, 
and $\cont(h)$ is the greatest common divisor of its coefficients; 

\item $d = \gcd(L(f), L(g))$.
\end{itemize}

	\section{About the resultant}
	In this section, we present some results about the resultant $\Res(f,g)$ which are needed later on.   
	
	For two integer polynomials $h_1(x)$ and $h_2(x)$, recall that their resultant $\Res(h_1,h_2)$ is defined to be 
	the determinant of their Sylvester matrix. 
	We first list three classical properties of resultants as follows (see \cite{Cox} or \cite{GKZ} for their proofs): 
	\begin{equation}\label{eq:hq}
		\Res(h_1,h_2)=(-1)^{\deg(h_1)\deg(h_2)} \Res(h_2,h_1),
	\end{equation}
	\begin{equation}\label{eq:hq1q2}
		\Res(h,h_1h_2)=\Res(h,h_1)\Res(h,h_2),
	\end{equation}
	\begin{equation}\label{eq:ch}
		\Res(c,h)=\Res(h,c)=c^{\deg(h)},
	\end{equation}
	where $h, h_1, h_2$ are integer polynomials of positive degree, and $c$ is a non-zero constant. 
	The result in \cite[Lemma 4.1]{Dilcher} asserts that 
	if we can write $h_2= sh_1+t$ with integer polynomials $s,t$, then 
	\begin{equation}\label{eq:qhL}
		\Res(h_1,h_2)=L(h_1)^{\deg(h_2)-\deg(t)}\Res(h_1,t).
	\end{equation}
	
	The following result except the divisibility part ``$d \mid \cont(p)$, $d \mid \cont(q)$" is well-known 
(see \cite[Proposition 4.18]{BPR} or \cite[Lemma 7.2.1]{Mishra}). 
We follow the approach in \cite[Lemma 7.2.1]{Mishra} to recover the well-known part and also prove the divisibility part. 
	
	\begin{theorem}\label{thm:uvR}
 There are two unique integer polynomials $p(x)$ and $q(x)$ such that $\deg(p) < \deg(g)$, $\deg(q) < \deg(f)$, $d \mid \cont(p)$, $d \mid \cont(q)$, and 
		\begin{equation*}
			p(x)f(x)+q(x)g(x) = R(f,g).
		\end{equation*}
		In particular, $d \mid R(f,g)$. 
	\end{theorem} 	
	\begin{proof}
		First, we write 
		\begin{align*}\label{def:f,g}
			&f(x)=a_{m}x^{m}+a_{m - 1}x^{m - 1}+\cdots+a_1x + a_{0},\\
			&g(x)=b_{n}x^{n}+b_{n - 1}x^{n - 1}+\cdots+b_1x + b_{0}, 
		\end{align*} 
		with $a_mb_n \ne 0$.  
		Regarding the Sylvester matrix of $f(x)$ and $g(x)$, for each $i$ ($1\leq i\leq m + n - 1$), we multiply the $i$-th column by $x^{m + n - i}$ and add it to the $(m + n)$-th column (here we view $x$ as an arbitrary non-zero integer). Then, we obtain
		$$
		\Res(f,g)= \det \left(\begin{array}{ccccccc}
			a_{m}&a_{m-1}&a_{m-2}&\cdots&0&0&x^{n-1}f\\
			0&a_{m}&a_{m-1}&\cdots&0&0&x^{n-2}f\\
			0&0&a_{m}&\cdots&0&0&x^{n-3}f\\
			\vdots&\vdots&\vdots&\vdots&\vdots&\vdots&\vdots\\
			0&0&0&\cdots&a_1&a_0&xf\\
			0&0&0&\cdots&a_2&a_1&f\\
			b_{n}&b_{n-1}&b_{n-2}&\cdots&0&0&x^{m-1}g\\
			0&b_{n}&b_{n-1}&\cdots&0&0&x^{m-2}g\\
			\vdots&\vdots&\vdots&\vdots&\vdots&\vdots&\vdots\\
			0&0&0&\cdots&b_1&b_0&xg\\
			0&0&0&\cdots&b_2&b_1&g
		\end{array}
		\right).$$
		In view of $d=\gcd(L(f), L(g))$, we can write $a_m=ad, b_n= bd$ for two coprime integers $a, b$. 
		Then, we obtain
		\begin{equation*}
			\begin{aligned}
				\Res(f,g)&
				=d \cdot \det \left(
				\begin{array}{ccccc}
					a&a_{m - 1}&\cdots&0&x^{n - 1}f\\
					0&a_{m}&\cdots&0&x^{n - 2}f\\
					\vdots&\vdots&\vdots&\vdots&\vdots\\
					0&0&\cdots&a_1&f\\
					b&b_{n - 1}&\cdots&0&x^{m - 1}g\\
					0&b_{n}&\cdots&0&x^{m - 2}g\\
					\vdots&\vdots&\vdots&\vdots&\vdots\\
					0&0&\cdots&b_1&g
				\end{array}
				\right). 
			\end{aligned}
		\end{equation*}
		We denote by $A$ the determinant on the right hand side of the above identity. 
		For $1 \le i, j \le m+n$, we denote the $(i,j)$-th cofactor of $A$ by $A_{ij}$. 
		Then, expanding $A$ via cofactors along the last column, we get
		\begin{align*}
			\Res(f,g)&=d \Big(x^{n-1}fA_{1,m+n}+ \cdots +fA_{n,m+n} +\\ 
			& \qquad \qquad x^{m-1}gA_{n+1,m+n}+ \cdots +gA_{m+n,m+n} \Big) \\ 
			&=d(x^{n-1}A_{1,m+n}+ \cdots +A_{n,m+n})f + \\ 
			& \qquad\qquad d(x^{m-1}A_{n+1,m+n}+ \cdots +A_{m+n,m+n})g\\
			&=p_1(x)f(x) + q_1(x)g(x),
		\end{align*} 
		where 
		\begin{align*}
			& p_1(x)=d(A_{1,m+n}x^{n-1}+ \cdots +A_{n,m+n}) \in \Z[x], \\
			& q_1(x) = d(A_{n+1,m+n}x^{m-1}+ \cdots +A_{m+n,m+n}) \in \Z[x].
		\end{align*}
		
		Taking absolute value, we obtain 
		\begin{equation}  \label{eq:RRes}
			R(f,g)=|\Res(f,g)|
			=p(x)f(x)+q(x)g(x),
		\end{equation}
		where $p(x)=p_1(x)$ and $q(x)=q_1(x)$ if $\Res(f,g)>0$, and otherwise $p(x)=-p_1(x)$ and $q(x)=-q_1(x)$. 

Note that in the above deductions, we view $x$ as an arbitrary non-zero integer. 
That is, \eqref{eq:RRes} holds for any non-zero integer $x$. 
So, \eqref{eq:RRes} also holds if we view $x$ as a variable. 
		
		Noticing that each cofactor  $A_{i, m+n}$ ($1 \le i \le m+n$) is an integer, 
		clearly we have 
		$$
		d \mid \cont(p), \qquad d \mid \cont(q). 
		$$

		Meanwhile, it can be seen that  $\deg(p)\leq m - 1<\deg(g)$, $\deg(q)\leq n - 1<\deg(f)$. 
		Finally, since $f$ and $g$ are coprime, one can prove by contradiction that $p(x)$ and $q(x)$ are unique. 	
	\end{proof}

	\section{Relations between $B(f,g)$ and $r(f,g)$}  
	We have already known that $r(f,g)\mid B(f,g)$ by \eqref{eq:rBR}. 
	The following result is about the converse of this divisibility relation. 
Recall that $d=\gcd(L(f), L(g))$.
	
	\begin{theorem}  \label{thm:Bdr}
		There exists a non-negative integer $k$ such that
		\begin{equation*}
			B(f,g) \mid d^k r(f,g). 
		\end{equation*}
	\end{theorem}
	
	\begin{proof}
		For simplicity, let $B = B(f,g)$ and $r = r(f,g)$. By definition, there exist two integer polynomials $p(x)$ and $q(x)$ such that
		\begin{equation}\label{eq:uvr}
			p(x)f(x) + q(x)g(x) = r. 
		\end{equation}
		Since $r \in \Z$, we must have 
		\begin{equation}  \label{eq:deg}
			\deg(p) + \deg(f) = \deg(q)+\deg(g).
		\end{equation} 
		So, either $\deg(p) < \deg(g)$ and $\deg(q) < \deg(f)$, or $\deg(p) \geq \deg(g)$ and $\deg(q) \geq \deg(f)$.

		If  $\deg(p) < \deg(g)$ and $\deg(q) < \deg(f)$, then we have 
		\begin{equation*}
			\frac{p(x)}{r}f(x) + \frac{q(x)}{r}g(x) = 1. 
		\end{equation*}
		According to the definition of $B$ (that is, $B(f,g)$), we have $B \mid r$.
		
		Now, we assume that $\deg(p) \geq \deg(g)$ and $\deg(q) \geq \deg(f)$. 
		From \eqref{eq:uvr} and noticing $r \in \Z$, we have $L(pf) = -L(qg)$, that is, 
		$$
		L(p)\cdot L(f) = -L(q)\cdot L(g).
		$$
		Moreover, we get
		\begin{equation}\label{Lk}
			L(p) \cdot \frac{L(f)}{d}=-L(q)\cdot \frac{L(g)}{d}.
		\end{equation}
		Since $\gcd(\frac{L(f)}{d},\frac{L(g)}{d}) = 1$, we have $\frac{L(f)}{d} \mid L(q)$ and $\frac{L(g)}{d} \mid L(p)$. Then, we define two integer polynomials: 
		\begin{equation*}
			h(x)=\frac{dL(p)}{L(g)}x^{\deg(p)-\deg(g)}, \qquad 
			t(x)=\frac{dL(q)}{L(f)}x^{\deg(q)-\deg(f)}.
		\end{equation*}
		By \eqref{eq:deg} and \eqref{Lk}, we obtain 	
		\begin{equation}  \label{eq:p1q1}
			h(x)+t(x)=0.
		\end{equation}
		
		So, we have
		\begin{align*}
			& d\cdot p(x)=g(x)h(x)+p_{1}(x),\\
			& d\cdot q(x)=f(x)t(x)+q_{1}(x),
		\end{align*}
		where $p_{1}(x)$ and $q_{1}(x)$ are some integer polynomials such that $\deg(p_1) < \deg(p)$ and $\deg(q_1)<\deg(q)$. Combining the above two formulas with \eqref{eq:uvr}, we get
		$$f(x)g(x)(h(x)+t(x))+p_{1}(x)f(x)+q_{1}(x)g(x)=dr,$$
		which, together with \eqref{eq:p1q1}, gives 
		$$p_{1}(x)f(x)+q_{1}(x)g(x)=dr.$$
		Repeating the above steps, after a finite number of steps (say $k$ steps), there must exist integer polynomials $p_k(x)$ and $q_k(x)$ such that $\deg(p_k) < \deg(g)$, $\deg(q_k) < \deg(f)$, and
		$$p_k(x)f(x)+q_k(x)g(x)=d^k r.$$
		Hence, as before we obtain 
		\begin{equation*}
			B \mid d^k r. 
		\end{equation*}
		This completes the proof. 
	\end{proof}
	
	We remark that the non-negative integer $k$ in Theorem~\ref{thm:Bdr} is at most $\max(\deg(p)-\deg(g)+1, 0)$ with polynomial $p$ in \eqref{eq:uvr}. So, in Theorem~\ref{thm:Bdr} we can choose 
$$
k = \max(\deg(p)-\deg(g)+1, 0).
$$
Moreover, one can find an upper bound on $\deg(p)$ for one choice of $p(x)$ in \cite[Theorem A]{Asch}. 
	
	When $d =1$ in Theorem~\ref{thm:Bdr}, we have $B(f,g) \mid r(f,g)$. Then, noticing $r(f,g)\mid B(f,g)$, we directly get the following corollary. 
We believe that this result is known to the experts (see, for instance, Myerson's answer in \cite{Rug} for the case of monic polynomials), but we haven't found a specific reference for it. 
	
	\begin{corollary}\label{cor:Br}
		If $d =1$, then 
		$B(f,g) = r(f,g)$.
	\end{corollary}

	If $d >1$, 
	then either $B(f,g) = r(f,g)$ or $B(f,g) \ne r(f,g)$. Both cases can happen; see Examples~\ref{ex:Br1} and \ref{ex:Br2}. 

	\begin{example}  \label{ex:Br1}
		{\rm 
			Let $f(x) = 6x^3 - 6x^2 - 6x-6$ and $g(x) = 6x^3 - 6x^2 - 6x + 5$. 
			Then, $B(f,g)=11$, $R(f,g) = 287496$ and $r(f,g) = 11$, and so, $R(f,g) =6^3 r(f,g)^3$.
		}
	\end{example}
	
	\begin{example} \label{ex:Br2}
		{\rm 
			Let $f(x) = 6x^{2}+5$ and $g(x) =  6x^{2}-4x + 1$. 
			Then, $B(f,g) = 22 = 2 \times 11$, $R(f,g) = 1056 = 2^5 \times 3 \times 11$ and $r(f,g) = 11$.
		}
	\end{example}

	We also remark that if $d >1$, then $B(f,g)$ and $r(f,g)$ may not have the same 
prime factors, and also $B(f,g)$ and $dr(f,g)$ may not have the same prime factors; see Example~\ref{ex:Br2}. 
However, by Theorem~\ref{thm:Bdr} we know that any prime factor of $B(f,g)$ divides $dr(f,g)$.

	\section{Relations between $r(f,g)$ and $R(f,g)$}
	We have already known that $r(f,g)\mid R(f,g)$ by \eqref{eq:rBR}. 
	The following result is about the converse of this divisibility relation. 
Recall that $d = \gcd(L(f), L(g))$.
	
	\begin{theorem}\label{thm:rR}
		We have 
		$$R(f,g) \mid d^j r(f,g)^{\max(m,n)},$$
		where $ m = \deg(f), n = \deg(g), j = \deg(pf)=\deg(qg)$ with polynomials $p,q$ in \eqref{eq:uvr}. 
Moreover, if both $f$ and $g$ are monic, then we have 
		$$R(f,g) \mid r(f,g)^{\min(m,n)}.$$
	\end{theorem}
	
	\begin{proof}
		Let $R = R(f,g)$ and $r = r(f,g)$. Since $r$ is a non-zero integer, by \eqref{eq:ch} we have 
		$$
		\Res(f,r)= r^{\deg(f)} = r^m.
		$$
		In addition, by \eqref{eq:uvr}, there exist two integer polynomials $p,q$ such that 
$$
pf+qg=r,
$$
 which, together with \eqref{eq:hq1q2} and \eqref{eq:qhL}, gives 
		\begin{equation*}
			\begin{split}
				\Res(f,r)&= \Res(f,pf+qg) = L(f)^{0-j}\Res(f,qg) \\
				&= L(f)^{-j}\Res(f,g)\Res(f,q),
			\end{split}
		\end{equation*}
		where $j =\deg(qg)= \deg(pf)$. So, we have 
		$$\Res(f,r)=r^m = L(f)^{-j}\Res(f,g)\Res(f,q),$$ 
		which implies 
		\begin{equation}\label{eq:Rr1}
			L(f)^{j}r^m = \Res(f,g)\Res(f,q).
		\end{equation}
		
		Similarly, we obtain 
		\begin{equation}\label{eq:Rr2}
			L(g)^{j}r^n = \Res(g,f)\Res(g,p).
		\end{equation}
		
		Note that by definition and by \eqref{eq:hq}, we have  
$$
R =|\Res(f,g)|=|\Res(g,f)|.
$$ 
		In view of $d=\gcd(L(f), L(g))$, we write $L(f)=ad, L(g)=bd$ for two coprime integers $a, b$ (that is $\gcd(a,b)=1$).  Then, combining \eqref{eq:Rr1} with \eqref{eq:Rr2}, we get 
		\begin{equation}  \label{eq:Rr3}
			R \mid a^j d^j r^m, \qquad R \mid b^j d^j r^n.
		\end{equation} 
		
		Now, for any prime factor $\ell$ of $R$, if $\ell^i \mid R$ and $\ell^{i+1} \nmid R$ for some positive integer $i$, 
		then by \eqref{eq:Rr3} and noticing  $\gcd(a,b)=1$, 
		we must have $\ell^i \mid d^j r^{\max(m,n)}$. 
		Hence, we have 
		$$
		R \mid d^j r^{\max(m,n)}. 
		$$
		This gives the first part of the theorem.  

Moreover, if both $f$ and $g$ are monic, then by \eqref{eq:Rr3} we have 
\begin{equation*}  
			R \mid r^m \quad \textrm{and} \quad R \mid r^n, 
		\end{equation*} 
and so, $R \mid r^{\min(m,n)}$. This completes the proof. 
	\end{proof}

We remark that for the general case in Theorem~\ref{thm:rR}, the power $\max(m,n)$ can not be replaced by $\min(m,n)$; 
see Example~\ref{ex:rR1} (for $d=1$) and Example~\ref{ex:rR2} (for $d>1$).

\begin{example} \label{ex:rR1}
		{\rm 
			Let $f(x) = 2x^{3}+3x^2-2$ and $g(x) =  3x-3$. 
			Then, $d=1$, $B(f,g) = r(f,g)=9 $, and $R(f,g) = 81$.
		}
\end{example}

\begin{example} \label{ex:rR2}
		{\rm 
			Let $f(x) = 2x^{3}+x^2-3x+2$ and $g(x) =  4x-2$. 
			Then, $d=2$, $B(f,g) = r(f,g)=2$,  and $R(f,g) = 64$. 
Moreover, in this case we can choose $j=3$, since 
$$
2f(x) - (x^2+x-1)g(x) = 2. 
$$
		}
\end{example}

	If $f(x)$ is monic (that is $L(f)=1$), by \eqref{eq:Rr1} we recover a special case of a result in \cite{Myerson1983}:
	
	\begin{corollary}\label{cor:rR}
		If moreover $f(x)$ is monic, then we have $R(f,g) \mid r(f,g)^m$, where $m=\deg(f)$.
	\end{corollary}
	
	By Theorem~\ref{thm:rR} and noticing $r(f,g) \mid R(f,g)$ and $d \mid R(f,g)$ (see Theorem~\ref{thm:uvR}), we directly get that $R(f,g)$ and $dr(f,g)$ have the same prime factors. 
This result is known to the experts. 

	\begin{corollary}  \label{cor:Rdr}
		$R(f,g)$ and $dr(f,g)$ have the same prime factors. 
	\end{corollary}
	
	We remark that when $d>1$, it can happen that $R(f,g)$ and $r(f,g)$ do not have the same prime factors; 
see Example~\ref{ex:Br1}.

	\section{Relations between $B(f,g)$ and $R(f,g)$}
	It is known to the experts that $B(f,g) \mid R(f,g)$. 
	We enhance this divisibility relation in the following result. 
Recall that $d = \gcd(L(f), L(g))$.
	
	\begin{theorem}\label{thm:BR1}
		We have 
		$$
		dB(f,g) \mid R(f,g).
		$$
	\end{theorem} 
	
	\begin{proof}
		Let $B = B(f,g)$ and $R = R(f,g)$. First, from \eqref{eq:uvB} we know that there exist uniquely two integer polynomials $p_{1}(x)$ and $q_{1}(x)$ such that $\deg(p_{1}) < \deg(g)$, $\deg(q_{1}) < \deg(f)$ and 
		\begin{equation}\label{eq:uvB1}
			p_{1}(x)f(x)+q_{1}(x)g(x)=B.
		\end{equation}
		By the definition of $B$, we have 
		\begin{equation}   \label{eq:u1v1}
			\gcd(\cont(p_1), \cont(q_1))=1.
		\end{equation}

		In addition, by Theorem~\ref{thm:uvR} there exist uniquely two integer polynomials $p_{2}(x)$ and $q_{2}(x)$ such that $\deg(p_{2})<\deg(g)$, $\deg(q_{2})<\deg(f)$ and 
		\begin{equation}\label{eq:uvR1}
			p_{2}(x)f(x)+q_{2}(x)g(x)=\frac{R}{d}.
		\end{equation}
		
	    Let $R_0=R/d$. Then, from \eqref{eq:uvB1} and \eqref{eq:uvR1}, we obtain 
		\begin{eqnarray*}
			&\frac{p_{1}(x)}{B}f(x)+\frac{q_{1}(x)}{B}g(x)=1,\\
			&\frac{p_{2}(x)}{R_0}f(x)+\frac{q_{2}(x)}{R_0}g(x)=1.
		\end{eqnarray*}
		By the uniqueness of the polynomials $p(x), q(x)$ in \eqref{eq:Bezout}, we must have 
		$$
		\frac{p_{1}(x)}{B}=\frac{p_{2}(x)}{R_0} \quad \textrm{and} \quad \frac{q_{1}(x)}{B}=\frac{q_{2}(x)}{R_0}. 
		$$
		So, we obtain 
		$$
		B \cont(p_2) = R_0 \cont(p_1), \qquad B \cont(q_2) = R_0 \cont(q_1). 
		$$
		If $B \nmid R_0$, then there is a prime number $\ell$ such that $\ell \mid B$, $\ell \mid \cont(p_1)$ and $\ell \mid \cont(q_1)$, 
		and so this contradicts with \eqref{eq:u1v1}.
		Hence, we must have $B \mid R_0$.
		This completes the proof. 
	\end{proof}
	
	When $d >1$, by Theorem~\ref{thm:BR1} we directly get $B(f,g) \ne R(f,g)$, and then also $r(f,g) \ne R(f,g)$.
	
	\begin{corollary}\label{cor:BRr1}
		If $d>1$, then $ B(f,g) \ne R(f,g) $ and $ r(f,g) \ne R(f,g)$. 
	\end{corollary}
	
	If $d=1$, 
	then either $B(f,g) = R(f,g)$ or $B(f,g) \ne R(f,g)$. Both cases can happen; see the following two examples. 
	
	\begin{example}  \label{ex:BrR1}
		{\rm 
			Let $f(x) = 2x^3 - x^2 - x$ and $g(x) = x^3 - x^2 +x + 1$. 
			Then, $B(f,g)=r(f,g)=R(f,g) = 2$.
		}
	\end{example}
	
	\begin{example}   \label{ex:BrR2}
		{\rm 
			Let $f(x) = 2x^3 + x^2 - x-1$ and $g(x) = x^3 - x^2 +x + 1$. 
			Then, $B(f,g)=r(f,g)=3$ and $R(f,g) = 27$.
		}
	\end{example}
	
	Moreover, combining Corollary~\ref{cor:Br} with Corollary~\ref{cor:BRr1} and noticing $r(f,g) \le B(f,g) \le R(f,g)$, we directly obtain the following result. 
	
	\begin{corollary}  \label{cor:BRr}
	   $B(f,g) = R(f,g)$ if and only if $r(f,g) = R(f,g)$.
	\end{corollary}
	
	Now, we establish a converse of the divisibility relation in Theorem~\ref{thm:BR1} as follows. 
	
	\begin{theorem} \label{thm:BR2}
		We have 
		$$R(f,g) \mid d^{m+n-1} B(f,g)^{\max(m,n)},$$
		where $m = \deg(f)$ and $n=\deg(g)$. 
Moreover, if both $f$ and $g$ are monic, then we have 
		$$R(f,g) \mid B(f,g)^{\min(m,n)}.$$
	\end{theorem}
	
	\begin{proof}
		Let $B = B(f,g)$ and $R = R(f,g)$. We have known that there exist uniquely two integer polynomials 
		$p_1(x)$ and $q_1(x)$ such that $\deg(p_1) < \deg(g)$, $\deg(q_1) < \deg(f)$ and 
		\begin{equation*}
			p_1(x)f(x)+q_1(x)g(x)=B.
		\end{equation*}
		Then, applying the same arguments as in the proof of Theorem~\ref{thm:rR}, we obtain 
		$$
		R \mid d^j B^{\max(m,n)}, 
		$$
		where $j=\deg(p_1f)=\deg(q_1g) \le m+n-1$. 
Also, when both $f$ and $g$ are monic, we have 
		$$R \mid B(f,g)^{\min(m,n)}.$$
		This completes the proof. 
	\end{proof}

We remark that for the general case in Theorem~\ref{thm:BR2}, the power $\max(m,n)$ can not be replaced by $\min(m,n)$; 
see Example~\ref{ex:rR1} (for $d=1$) and Example~\ref{ex:rR2} (for $d>1$).
	
	By Theorems~\ref{thm:BR1} and \ref{thm:BR2}, we directly get that $R(f,g)$ and $dB(f,g)$ have the same prime factors. This also can be seen from Corollary~\ref{cor:Rdr}. 

	\begin{corollary} \label{cor:RdB}
	   $R(f,g)$ and $dB(f,g)$ have the same prime factors. 
	\end{corollary}
	
	We remark that when $d>1$, it can happen that $B(f,g)$ and $R(f,g)$ do not have the same prime factors; 
see Example~\ref{ex:Br2}. 
	
	Combining Corollary~\ref{cor:Rdr} with Corollary~\ref{cor:RdB}, we directly get the following corollary. 
We also believe that this result is known to the experts, but we haven't found a specific reference for it. 
	
	\begin{corollary}
		If $d=1$, then 
		$B(f,g)$, $r(f,g)$ and $R(f,g)$ have the same prime factors. 
	\end{corollary}

	\section{Numerical data}
	
	In this section, we present some numerical data to see how often we have $B(f,g)=r(f,g)$  
	(or moreover $R(f,g)=r(f,g)$) for two coprime integer polynomials $f(x), g(x)$ which are given randomly. 
	
	First, for an integer polynomial $h(x) \in \Z[x]$, we denote by $\wH(h)$  the maximum of the absolute values of its coefficients. Recall that $L(h)$ stands for the leading cofficient of $h$. 
	
	Then, for three positive integers $m, n$ and $H$, we define the following set of pairs of coprime integer polynomials:
	\begin{align*}
		S_{m,n}(H) = & \Big\{(f,g) \in \Z[x]^2: \, \deg (f)=m, \deg (g) = n,  L(f)>0, \\
		& \quad L(g)>0, \textrm{$f$ and $g$ are coprime},  \wH(f) \le H, \wH(g) \le H \Big\}.
	\end{align*}
	Note that in the set $S_{m,n}(H)$, both $f$ and $g$ are assumed to have positive leading coefficients, 
	because 
	$$
	f(x)\Z[x] + g(x)\Z[x] = (\pm 1)f(x)\Z[x] + (\pm 1)g(x)\Z[x]. 
	$$

	In Table~\ref{tab:Br}, each value corresponds to $m, n, H$ is the following percentage:
	\begin{equation*}
		\frac{|\{(f,g)\in S_{m,n}(H): \, B(f,g) = r(f,g)\}|}{|S_{m,n}(H)|} \times 100\%, 
	\end{equation*}
	which are rounded to the nearest two decimal places in the table. 
	Recall that the probability that two randomly given integers are coprime is $6/\pi^2 \approx 0.6079$; see \cite[Theorem 332]{Hardy}. So, by Corollary~\ref{cor:Br}, we can roughly say that the above percentage is at least $60\%$ when $H$ is sufficiently large. However, Table~\ref{tab:Br} shows that this percentage is much larger (more than $90\%$). 
Moreover, when $(m,n)=(1,1)$, this percentages corresponding to $H=30, 50$ and $100$ are respectively: 
$90.29\%$, $90.32\%$, and $90.11\%$.

	Hence, we pose the following conjecture. 	It suggests that for two coprime integer polynomials $f(x), g(x)$ given randomly, we usually have $B(f,g) = r(f,g)$. 
	
	\begin{conjecture}
		For any positive integers $m$ and $n$, we have 
		\begin{equation*}
			\lim_{H \to \infty}\frac{|\{(f,g)\in S_{m,n}(H): \, B(f,g) = r(f,g)\}|}{|S_{m,n}(H)|} \ge \frac{9}{10}. 
		\end{equation*}
	\end{conjecture}

	\begin{table}
		\begin{tabular}{|c|c|c|c|c|c|} \hline
			\diagbox{$(m, n)$}{$H$} & 2 & 3	& 4	& 5 & 6  \\ \hline
			(1, 1) & 97.44\% & 95.90\% & 93.61\% & 94.73\% & 92.13\%   \\ \hline
			(1, 2) & 98.58\% & 98.04\% & 96.53\% & 97.48\% & 95.98\%  \\ \hline
			(1, 3) & 99.36\% & 99.09\% & 98.13\% & 98.81\% &98.04\%   \\ \hline
			(2, 2) & 99.00\% & 99.10\% & 98.08\% & 98.78\% & 98.09\%    \\ \hline
			(2, 3) & 99.52\% & 99.55\% & 99.04\% & 99.43\% & 99.03\%    \\ \hline
			(3, 3) & 99.69\% & 99.78\% & 99.51\% & 99.72\% & 99.51\%  \\ \hline
		\end{tabular}
		\vspace{0.4cm}
		\caption{Percentage of $B(f,g)=r(f,g)$}
		\label{tab:Br}
	\end{table}

	From Corollary~\ref{cor:BRr}, we know that the frequency of $B(f,g)=R(f,g)$ is the same as that of $r(f,g) = R(f,g)$.
	So, we only need to consider how often we have $B(f,g) = R(f,g)$. 
	
	In Table~\ref{tab:BR}, each value corresponds to $m, n, H$ is the following percentage:
	\begin{equation*}
		\frac{|\{(f,g)\in S_{m,n}(H): \, B(f,g) = R(f,g)\}|}{|S_{m,n}(H)|} \times 100\%,
	\end{equation*}
	which are rounded to the nearest two decimal places in the table. 
	As the above, by Corollary~\ref{cor:BRr1} we can roughly say that this percentage is at most $61\%$ when $H$ is sufficiently large. However, Table~\ref{tab:BR} shows that this percentage is much smaller except the case when $(m,n)=(1,1)$. 

	Table~\ref{tab:BR} suggests that for two coprime integer polynomials $f(x), g(x)$ given randomly, 
	we often have $B(f,g) \ne R(f,g)$ when $\deg (f)$ and $\deg (g)$ are large. 
	Regarding Table~\ref{tab:BR}, we pose the following conjecture. 
	
	\begin{conjecture}
		For any positive integers $m$ and $n$ with $m \ge 2$ and $n \ge 2$, we have 
		\begin{equation*}
			\lim_{H \to \infty}\frac{|\{(f,g)\in S_{m,n}(H): \, B(f,g) = R(f,g)\}|}{|S_{m,n}(H)|} < \frac{1}{2}. 
		\end{equation*}
	\end{conjecture}

Given two coprime integer polynomials of degree one $f(x)=ax+b$ and $g(x)=cx+d$, 
we directly have $R(f,g) = |ad -bc|$ and 
$$
\frac{c}{bc - ad} \cdot f(x) - \frac{a}{bc-ad} \cdot g(x) = 1. 
$$
So, by definition we have $B(f, g) = |ad -  bc|/\gcd(a,c)$. 
Hence, $B(f,g) = R(f,g)$ whenever $\gcd(a,c)=1$. 
This can explain why the percentages in Table~\ref{tab:BR} corresponding to the case when $(m,n)=(1,1)$ are large.

	\begin{table}
		\begin{tabular}{|c|c|c|c|c|c|} \hline
			\diagbox{$(m, n)$}{$H$} & 2 & 3	& 4	& 5 & 6   \\ \hline
			(1, 1) & 82.05\% & 81.54\% & 75.31\% & 79.47\% & 70.35\%   \\ \hline
			(1, 2) & 64.15\% & 61.14\% & 51.73\% &	58.16\% & 46.13\%   \\ \hline
			(1, 3) & 62.80\% & 60.44\% & 50.56\% & 57.52\% & 45.04\%  \\ \hline
			(2, 2) & 46.14\% & 49.81\% & 41.47\% & 49.08\% & 37.76\%  \\ \hline
			(2, 3) & 46.61\% & 49.54\% & 41.33\% & 48.66\% & 37.66\% \\ \hline
			(3, 3) & 46.29\% & 49.24\% & 41.24\% & 48.25\% & 37.74\%  \\ \hline
		\end{tabular}
		\vspace{0.4cm}
		\caption{Percentage of $B(f,g)=R(f,g)$}
		\label{tab:BR}
	\end{table}

\section{Comments}
	In this section, we want to say some more words about the resultant and relevant concepts. 

The resultant of two univariate polynomials over the integers $\Z$ is an important concept and useful in 
algebra and number theory. 
It has been generalized to a more general setting (see \cite[Section 4.2]{BPR} or \cite[Chapter 7]{Mishra}) and also 
to multivariate polynomials (see \cite[Chapter 3, Section 6]{Cox}). 
One can also interpret it from the viewpoint of algebraic geometry; see \cite[Section V.2]{EH}. 
From this viewpoint, one may define new integers with respect to two coprime integer polynomials. 
 
Moreover,  the subresultants of two uivariate integer polynomials 
are defined by means of submatrices of their Sylvester matrix 
(see  \cite[Section 4.2]{BPR} or \cite[Chapter 7]{Mishra}  for a more general setting). 
The subresultants can be viewed as a generalization of the resultant, and they are very useful for computing the resultant and the greatest common divisor of those two polynomials.  

In addition, since the polynomial ring $\Z[x]$ is a unique factorization domain, for the two coprime polynomials $f,g \in \Z[x]$ given above, their greatest common divisor $\gcd(f,g)$ exists in $\Z[x]$ and in fact equals to $\gcd(\cont(f), \cont(g))$ (that is, the greatest common divisor of all the coefficients of $f$ and $g$). So, the integer $\gcd(f,g)$ is usually different from the three integers 
$B(f,g)$, $r(f,g)$ and $R(f,g)$ we have studied.

	\section*{Acknowledgement}
%	The authors would like to thank the referees for their valuable comments. 
	For the research, Z. Liu and X. Li were supported by the Guangdong College Students' 
	Innovation and Entrepreneurship Training Program (No. S202410574067); 
	M. Sha was supported by the Guangdong Basic and Applied Basic Research Foundation (No. 2025A1515010635).

\end{document}